\documentclass[letterpaper, 10 pt, conference]{ieeeconf}  

\IEEEoverridecommandlockouts                              
\usepackage{bbm}
\usepackage{mathrsfs}
\usepackage{verbatim}
\usepackage{amsmath}
\usepackage{amsfonts}
\usepackage{bm}
\usepackage{graphicx}
\usepackage{float}

\usepackage{amsthm}
\usepackage{xcolor}
\usepackage[colorinlistoftodos]{todonotes}
\usepackage{mathtools}
\usepackage{cite}
\usepackage{colortbl}
\usepackage{subfigure}
\allowdisplaybreaks

\theoremstyle{definition}
\newtheorem{assumption}{Assumption}
\newtheorem{definition}{Definition}
\newtheorem{remark}{Remark}
\newtheorem{problem}{Problem}

\theoremstyle{plain}

\newtheorem{theorem}{Theorem}
\newtheorem{lemma}{Lemma}

\title{\LARGE \bf On Robust Control of Partially Observed Uncertain Systems with Additive Costs}

\author{Aditya Dave, {\itshape{Student Member, IEEE,}} Nishanth Venkatesh, {\itshape{Student Member, IEEE,}}  \\
Andreas A. Malikopoulos, {\itshape{Senior Member, IEEE}} 
	\thanks{This research was supported by NSF under Grants CNS-2149520 and CMMI-2219761. 
	The authors are with the Department of Mechanical Engineering, University of Delaware, Newark, DE 19716 USA (email: \texttt{adidave@udel.edu; nish@udel.edu; andreas@udel.edu).}} }

\begin{document}

\maketitle
\thispagestyle{empty}

\begin{abstract}
In this paper, we consider the problem of optimizing the worst-case behavior of a partially observed system. All uncontrolled disturbances are modeled as finite-valued uncertain variables. 
Using the theory of cost distributions, we present a dynamic programming (DP) approach to compute a control strategy that minimizes the maximum possible total cost over a given time horizon. 
To improve the computational efficiency of the optimal DP, we introduce a general definition for information states and show that many information states constructed in previous research efforts are special cases of ours. Additionally, we define approximate information states and an approximate DP that can further improve computational tractability by conceding a bounded performance loss. 
We illustrate the utility of these results using a numerical example.
\end{abstract}

\section{Introduction}
\label{section:Introduction}

In engineering applications, it is common for an agent to operate with limited knowledge of the system state and uncertain system dynamics \cite{Malikopoulos2021}. This decision-making challenge is typically modeled as a stochastic control problem, where the agent computes a control strategy to minimizes an expected total cost across a time horizon given a prior probability distribution for all uncertainties. This approach has also been utilized in reinforcement learning \cite{subramanian2022approximate} and decentralized systems \cite{Dave2021nestedaccess}.
However, the expected total cost may not be an adequate measure of performance in all situations. In fact, many applications require guarantees on a system's worst-case performance, for example: (1) control of systems under attack from an adversary, like cyber-security systems \cite{rasouli2018scalable}, and (2) control of systems where a single event of failure can be damaging, like water reservoirs \cite{giuliani2021state}. Furthermore, the performance of a stochastic control strategy degrades rapidly with a mismatch between the assumed prior distribution and the actual underlying distribution \cite{mannor2007bias}. Consequently, stochastic models are unsuitable for strategy computation when prior distributions are ambigious. 


For such applications, we can instead utilize a non-stochastic formulation, where the agent only has access to the feasible sets for all uncertainties, without knowledge of probability distributions. This non-stochastic approach has been utilized in robust control \cite{james1994risk, bernhard2003minimax, bernhard1996separation}, information theory \cite{nair2013nonstochastic, rangi2019towards}, reinforcement learning \cite{gradu2020non, Dave2023approximate}, and decentralized systems \cite{gagrani2017decentralized, Dave2021minimax}. In this paper, we focus on a centralized non-stochastic control problem where an agent seeks a control strategy to minimize a \textit{maximum} possible cost over a finite-time horizon. It is known that the optimal strategy in such problems can be computed with an offline dynamic program (DP) \cite{bernhard2003minimax}. However, the growth in the agent's memory with time makes this challenging because the agent's action is a function of the memory and thus, the DP requires solving one optimization problem at each time for each possible realization of the memory. Using an \textit{information state} can address this challenge. Two well known non-stochastic information states are the \textit{conditional range} 
for terminal cost problems \cite{bertsekas1973sufficiently, bacsar2008h} and 
the \textit{maximum cost-to-come} 
for additive cost problems \cite{james1994risk, bernhard2003minimax}. In robust stochastic problems \cite{wiesemann2013robust} concerns of partial observation have also been addressed using a conditional range \cite{rasouli2018robust}.
Generalized approximate information states for terminal cost problems were developed in \cite{Dave2022approx}.
However, to the best of our knowledge, there is no notion of approximate information states for non-stochastic \textit{additive} cost problems.

The main contributions of this paper are: (1) for additive cost problems, we introduce general information states to compute an optimal strategy (Theorem \ref{thm_opt_ad_dp}), and (2) we define approximate information states to compute an approximate strategy with a bounded performance loss (Theorems \ref{thm_approx_ad_dp} - \ref{thm_approx_ad_policy}). 

The remainder of the paper proceeds as follows. In Section \ref{section:problem}, we present our model. In Section \ref{section:info_state}, we define information states and the corresponding DP. In Section \ref{section:approx}, we define approximate information states, the approximate DP, and derive performance bounds. In Section \ref{section:example}, we present a numerical example to illustrate our results.
Finally, in Section \ref{section:conclusion}, we draw concluding remarks and discuss ongoing work.

\section{Model}
\label{section:problem}

\subsection{Notation and Preliminaries}

We use the non-stochastic framework of \textit{uncertain variables} from \cite{nair2013nonstochastic}.
For a sample space $\Omega$ and a set $\mathscr{X}$, an uncertain variable is a mapping $\mathsf{X}: \Omega \to \mathscr{X}$ written concisely as $\mathsf{X} \in \mathscr{X}$. For any $\omega \in \Omega$, its realization is $\mathsf{X}(\omega) = \mathsf{x} \in \mathscr{X}$. 
The \textit{marginal range} of an uncertain variable $\mathsf{X}$ is the set $[[\mathsf{X}]] \hspace{-1pt} := \hspace{-1pt} \{\mathsf{X}(\omega) \; | \; \omega \in \Omega\}$. 
The \textit{joint range} of two uncertain variables $\mathsf{X} \in \mathscr{X}$ and $\mathsf{Y} \hspace{-1pt} \in \hspace{-1pt} \mathscr{Y}$ is $[[\mathsf{X},\mathsf{Y}]] \hspace{-1pt} := \hspace{-1pt} \{ (\mathsf{X}(\omega), \mathsf{Y}(\omega)) \; | \; \omega \in \Omega \}$. 
The \textit{conditional range} of $\mathsf{X}$ given a realization $\mathsf{y}$ of $\mathsf{Y}$ is $[[\mathsf{X}|\mathsf{y}]] \hspace{-1pt} := \{ \mathsf{X}(\omega) \; | \; \mathsf{Y}(\omega) = \mathsf{y}, $ $ \omega \in \Omega \}$, and $[[\mathsf{X}|\mathsf{Y}]] \hspace{-1pt} := \hspace{-1pt} \{ [[\mathsf{X}|\mathsf{y}]]\; | \; \mathsf{y} \in [[\mathsf{Y}]] \}$. 
Next, consider two compact, nonempty subsets $\mathscr{X}, \mathscr{Y}$ of a metric space $(\mathscr{S}, d)$, where $d(\cdot,\cdot)$ is the metric. Then, the Hausdorff distance \cite[Chapter 1.12]{barnsley2006superfractals} between the sets is
$\mathcal{H}(\mathscr{X}, \mathscr{Y}) \hspace{-2pt} := \hspace{-2pt} \max \hspace{-1pt} \{ \hspace{-1pt} \max_{\mathsf{x} \in \mathscr{X}} \hspace{-1pt} \min_{\mathsf{y} \in \mathscr{Y}}  d(\mathsf{x}, \mathsf{y}), \max_{\mathsf{y} \in \mathscr{Y}} \hspace{-1pt} \min_{\mathsf{x} \in \mathscr{X}}  d(\mathsf{x}, \mathsf{y}) \hspace{-2pt} \}.$

\subsection{Problem Formulation}

We consider an agent who controls the evolution of a system over $T \in \mathbb{N}$ discrete time steps. At any time $t=0,\dots, T$, the system is denoted by an uncertain variable $X_t \in \mathcal{X}$ and the agent's action is denoted by an uncertain variable $U_t \in \mathcal{U}$. At each $t$, the system also receives an uncontrolled disturbance $W_t \in \mathcal{W}$. Starting with an initial state $X_0 \in \mathcal{X}$, the state evolves as $X_{t+1}=f_t\left(X_t,U_t,W_t\right)$ for all $t=0,\dots,T-1$. Before selecting the control action at each $t$, the agent partially observes the system state as $Y_t=h_t(X_t,N_t) \in \mathcal{Y}$, where $N_t \in \mathcal{N}$ is a noise. 

\begin{remark}
We denote generic uncertain variables by sans-serif upper case alphabets $\mathsf{X} \in \mathscr{X}$ and $\mathsf{Y} \in \mathscr{Y}$, whereas, we denote the state and observation at any $t$ by italicized upper-case alphabets $X_t \in \mathcal{X}$ and $Y_t \in \mathcal{Y}$, respectively.
\end{remark}

At each $t=0,\dots,T$, the agent stores the history of observations and control actions in their memory, denoted by $M_t := (Y_{0:t}, U_{0:t-1}) \in \mathcal{M}_t$, where $Y_{0:t} := (Y_0,\dots,Y_t)$. Then, the agent selects an action
$U_t = g_t(M_t)$ using a control law $g_t: \mathcal{M}_t \to \mathcal{U}$ and incurs a cost $c_t(X_t, U_t) \in \mathbb{R}_{\geq0}$. We denote the control strategy by $\boldsymbol{g} := (g_0,\dots,g_T) \in \mathcal{G}$ and measure its performance using the \textit{worst-case criterion:}
\begin{align} \label{eq_ad_criterion}
    \mathcal{J}(\boldsymbol{g}) := \max_{\substack{x_0 \in \mathcal{X}, n_{0:T} \in \mathcal{N}^T, \\ w_{0:T-1} \in \mathcal{W}^{T-1}}} \sum_{t=0}^T c_t(X_t, U_t).
\end{align}
In \eqref{eq_ad_criterion}, we maximize the total cost over all feasible realizations of the \textit{uncontrolled inputs}, i.e., initial state $X_0$, noises $\{N_t~|~ t=0,\dots,T\}$, and disturbances $\{W_t~|~ t=0,\dots,T-1\}$ because they determine all other variables in the system. 
Next, we state the agent's optimization problem.

\begin{problem} \label{problem_1}
We seek to efficiently compute an optimal strategy $\boldsymbol{g}^* = \arg\min_{\boldsymbol{g} \in \mathcal{G}} \mathcal{J}(\boldsymbol{g}),$
given the sets $\{\mathcal{X}, \mathcal{U}, \mathcal{Y}, \mathcal{W}, \mathcal{N}\}$ and the functions $\{f_t, h_t, c_t~|~ t=0,\dots,T\}$.
\end{problem}

We impose the following assumptions on our model:

\begin{assumption} \label{assumption_primitive}
Each uncontrolled input is independent (see \cite[Definintion 2.1]{nair2013nonstochastic}) of all other uncontrolled inputs.
\end{assumption}

Assumption \ref{assumption_primitive} ensures that the system evolution is Markovian in a non-stochastic sense (see \cite[Definintion 2.2]{nair2013nonstochastic}). This assumption will help develop our results.

\begin{assumption} \label{assumption_finitie}
Each feasible set $\{\mathcal{X}, \mathcal{U}, \mathcal{Y}, \mathcal{W}, \mathcal{N}\}$ is a finite subset of a metric space $(\mathcal{S}, d)$.
\end{assumption}

Assumption \ref{assumption_finitie} ensures that all extrema are well defined and that an optimal solution to Problem \ref{problem_1} exists. We will use the metric $d(\cdot, \cdot)$ in Section IV to quantify the distance between two elements in any set.



\begin{assumption} \label{assumption_2}

All uncertain variables and the cost $c_t(X_t, U_t)$ have a finite maximum value at each $t$.

\end{assumption}

Assumption \ref{assumption_2}, in addition to the finiteness of all feasible sets, ensures that the functions $\{f_t, h_t, c_t ~|~ t=0,\dots,T\}$ are globally Lipschitz . To this end, we will denote the Lipschitz constant of a function $f_t$ by $L_{f_t} \in \mathbb{R}_{\geq0}$.

\section{Dynamic Programs and Information States} \label{section:info_state}

In this section, we first present a standard terminal cost DP which can obtain the optimal strategy in Problem \ref{problem_1}. Then, in Subsection \ref{subsection:specialized_DP}, we construct a DP which is specialized to the additive cost criterion in \eqref{eq_ad_criterion}, and in Subsection \ref{subsection:info_state}, we define information states to simplify it. To begin, we transform Problem \ref{problem_1} into a terminal cost problem by augmenting the state $X_t$ at each $t$ with the \textit{accrued cost}
\begin{gather} \label{accrued_def}
    A_t := \sum_{\ell = 0}^{t-1} c_\ell(X_\ell, U_{\ell}),
\end{gather}
which takes values in a finite set $\mathcal{A}_t \subset \mathbb{R}_{\geq0}$. Starting with $A_0 := 0$, the accrued cost evolves as $A_{t+1} = A_t + c_t(X_t, U_t)$ for all $t=0,\dots,T-1$. Thus, the augmented state $(X_t, A_t)$ evolves as a controlled Markov chain. Furthermore, note that the performance criterion \eqref{eq_ad_criterion} can be written as a function of the terminal augmented state $(X_T, A_T)$, i.e., $\mathcal{J}(\boldsymbol{g}) = \max_{x_0, n_{0:T}, w_{0:T-1}} \big(c_T(X_T, U_T) + A_T\big)$. This construction yields a terminal cost optimization problem in $\boldsymbol{g} \in \mathcal{G}$, where the optimal strategy can be computed using a memory based terminal cost DP \cite{Dave2022approx}, as follows. For all $m_t \in \mathcal{M}_t$ and $u_t \in \mathcal{U}$, for all $t=0,\dots,T-1$, we define the value functions
\begin{gather}
    Q_t^{\text{tm}}(m_t,u_t) := \max_{m_{t+1} \in [[M_{t+1}|m_t, u_t]]} V_{t+1}^{\text{tm}}(m_{t+1}), \label{DP_ad_basic_1} \\
    V_t^{\text{tm}}(m_t) := \min_{u_t \in \mathcal{U}} Q_t^{\text{tm}}(m_t,u_t), \label{DP_ad_basic_2}
\end{gather}
where, at time $T$, $Q_T^{\text{tm}}(\hspace{-1pt}m_T,u_T\hspace{-1pt}) \hspace{-2pt} := \hspace{-2pt} \max_{a_T, x_T \in [[A_T, X_T|m_T, u_T]]}$ $\big(c_T(x_T,u_T) + a_T\big)$ and $V_T^{\text{tm}}(m_T) := \min_{u_T \in \mathcal{U}} Q_T(m_T,u_T)$. The control law at each $t$ is $g_t^{\text{tm}}(m_t) := \arg \min_{u_t \in \mathcal{U}}$ $Q_t(m_t, u_t)$.
Using standard arguments, we can conclude that the resulting control strategy $\boldsymbol{g}^{\text{tm}} = (g_0^{\text{tm}},\dots,$ $g_T^{\text{tm}})$ is an optimal solution to the terminal cost problem as well as Problem \ref{problem_1} \cite{bertsekas1973sufficiently}. However, note that the right hand side (RHS) of \eqref{DP_ad_basic_2} involves solving a minimization problem for each possible realization $m_t \in \mathcal{M}_t$, at each $t$. The number of possible realizations $|\mathcal{M}_t|$ increases with time as the agent receives more observations, and consequently, the DP requires a large number of computations for a longer horizon $T$. To address this, we formulate a DP specialized for additive cost problems in Subsection \ref{subsection:specialized_DP} and simplify it using \textit{information states} in Subsection \ref{subsection:info_state}. We will show (Remark \ref{remark:improvement}) that the specialized DP allows us to define more computationally efficient information states than \eqref{DP_ad_basic_1} - \eqref{DP_ad_basic_2}.
To this end, we present a theory of cost distributions in the next subsection which is required to construct the specialized DP.

\subsection{Cost distributions} \label{subsection:cost_distributions}

In this subsection, we develop the mathematical framework of \textit{cost distributions} for finite uncertain variables. Cost distributions were originally defined for $(\max,+)$ algebra \cite{kolokoltsov1997idempotent}, and applied to robust control problems \cite{bernhard1996separation, bernhard2003minimax} independently from the framework of uncertain variables. A cost distribution is a non-stochastic analogue of a probability distribution.
Specifically, for a finite sample space $\Omega$ with a sigma algebra $\mathcal{B}(\Omega)$, a cost distribution is a function $q: \mathcal{B}(\Omega) \to \{-\infty\} \cup (-\infty,0]$ satisfying the properties: (1) $q(\Omega) = 0$, (2) $q(\emptyset) = - \infty$, and (3) $q(B) = \max_{\omega \in B} q(\omega)$ for all $B \in \mathcal{B}(\Omega)$, where, by convention, the maximum over an empty set is $- \infty$. Furthermore, for two sets $B^1, B^2 \in \mathcal{B}(\Omega)$ with $q(B^2) > - \infty$, the conditional cost distribution of $B^1$ given $B^2$ is 
$q(B^1|B^2) := q(B^1, B^2) - q(B^2),$ where $q(B^1, B^2) = \max_{\omega \in B^1 \cap B^2} q(\omega).$
Next, we extend this definition to include finite uncertain variables.

\begin{definition} \label{cdist_definition}
Let $\mathsf{X}: \Omega \to \mathscr{X}$ and $\mathsf{Y}: \Omega \to \mathscr{Y}$ be two finite uncertain variables. The \textit{cost distribution} for any realization $\mathsf{x} \in \mathscr{X}$ is
$q(\mathsf{x}) := \max_{\omega \in \{\Omega|\mathsf{X}(\omega) = \mathsf{x}\}} q(\omega)$, and that for any $\mathsf{x} \in \mathscr{X}$ given a realization $\mathsf{y} \in \mathscr{Y}$ with $q(\mathsf{y}) > - \infty$ is $q(\mathsf{x}|\mathsf{y}) = q(\mathsf{x},\mathsf{y}) - q(\mathsf{y})$, where $q(x,y) = \max_{\omega \in \{\Omega|X(\omega)=x, Y(\omega) = y\}} q(\omega)$.
\end{definition}


Any cost distribution given by Definition \ref{cdist_definition} satisfies the following useful properties. 

\begin{lemma} \label{lem_cdist_property}
Let $(\Omega, \mathcal{B}(\Omega))$ have a cost distribution $q: \mathcal{B}(\Omega) \to \{-\infty\} \cup (-\infty, 0]$. Let $\mathsf{X}: \Omega \to \mathscr{X}$ and $\mathsf{Y}: \Omega \to \mathscr{Y}$ be two finite uncertain variables and let $f: \mathscr{X} \to \mathscr{Y}$ such that $\mathsf{Y} = f(\mathsf{X})$ and $f^{-1}(\mathsf{y}) \neq \emptyset$ for all $\mathsf{y} \in \mathscr{Y}$. Then, 
\begin{gather} \label{eq_cdist_1}
    q(\mathsf{y}) = \max_{\mathsf{x} \in \{\mathscr{X}|f(\mathsf{x}) = \mathsf{y}\}} q(\mathsf{x}), \quad \forall \mathsf{y} \in \mathscr{Y},
\end{gather}
and furthermore, for any function $g : \mathscr{Y} \to \mathbb{R}_{\geq0}$, 
\begin{gather} \label{eq_cdist_2}
    \max_{\mathsf{x} \in \mathscr{X}} \big( g(f(\mathsf{x})) + q(\mathsf{x}) \big) = \max_{\mathsf{y} \in \mathscr{Y}} \big( g(\mathsf{y}) + q(\mathsf{y}) \big).
\end{gather}
\end{lemma}

\begin{proof}
Using Definition \ref{cdist_definition}, $q(\mathsf{y}) = \max_{\omega \in \{\Omega|\mathsf{Y}(\omega) = \mathsf{y}\}}$ $q(\omega)$, where $\{\Omega~|~\mathsf{Y}(\omega) = \mathsf{y}\} = \cup_{\mathsf{x} \in \{\mathscr{X}|f(\mathsf{x}) = \mathsf{y}\}} \{\Omega~|~ \mathsf{X}(\omega)  = \mathsf{x}\}$. This implies that $q(\mathsf{y}) = \max_{\mathsf{x} \in \{\mathscr{X}|f(\mathsf{x}) = \mathsf{y}\}} \max_{\omega \in \{\Omega|\mathsf{X}(\omega) = \mathsf{x}\}} $ $q(\omega) = \max_{\mathsf{x} \in \{\mathscr{X}|f(\mathsf{x}) = \mathsf{y}\}} q(\mathsf{x})$, where, in the second equality, we  used Definition \ref{cdist_definition}. This proves \eqref{eq_cdist_1}. Next, we use \eqref{eq_cdist_1} in the RHS of \eqref{eq_cdist_2} as $\max_{\mathsf{y} \in \mathscr{Y}} ( g(\mathsf{y}) + q(\mathsf{y}) )= \max_{\mathsf{y} \in \mathscr{Y}}( g(\mathsf{y}) + \max_{\mathsf{x} \in \{\mathscr{X}|f(\mathsf{x}) = \mathsf{y}\}}$ $q(\mathsf{x}) ) = \max_{\mathsf{y} \in \mathscr{Y}}\max_{\mathsf{x} \in \{\mathscr{X}|f(\mathsf{x}) = \mathsf{y}\}}( g(f(\mathsf{x}))+ q(\mathsf{x}) ) = \max_{\mathsf{x} \in \mathscr{X}}$  $( g(f(\mathsf{x})) + q(\mathsf{x}) )$, which completes the proof for \eqref{eq_cdist_2}.
\end{proof}


\subsection{Specialized Dynamic Program} \label{subsection:specialized_DP}

In this subsection, we construct a specialized DP decomposition for Problem \ref{problem_1} using two specific cost distributions, the first of which is an indicator function.

\begin{definition} \label{def_ind}
Let $\mathsf{X} \in \mathscr{X}$ and $\mathsf{Y} \in \mathscr{Y}$ be two finite uncertain variables. The \textit{indicator function} for any $\mathsf{x} \in \mathscr{X}$ is given by
\begin{gather}
    \mathbb{I}(\mathsf{x}) :=
    \begin{aligned}
    \begin{cases}
        0, &\text{ if } \mathsf{x} \in [[\mathsf{X}]], \\
        - \infty, &\text{ if } \mathsf{x} \not\in [[\mathsf{X}]],
    \end{cases}
    \end{aligned}
\end{gather}
and the conditional indicator function for any $\mathsf{x} \in \mathscr{X}$ given a realization $\mathsf{y} \in \mathscr{Y}$ with $\mathbb{I}(\mathsf{y}) > - \infty$ is
\begin{gather} \label{indicator_def}
    \mathbb{I}(\mathsf{x}|\mathsf{y}) :=
    \begin{aligned}
    \begin{cases}
        0, &\text{ if } \mathsf{x} \in [[\mathsf{X}|\mathsf{y}]], \\
        - \infty, &\text{ if } \mathsf{x} \not\in [[\mathsf{X}|\mathsf{y}]].
    \end{cases}
    \end{aligned}
\end{gather}
\end{definition}

The indicator function $\mathbb{I}$ can be shown to satisfy the conditions in Definition \ref{cdist_definition} and thus, it constitutes a valid cost distribution. In addition to Lemma \ref{lem_cdist_property}, for two uncertain variables $\mathsf{X} \in \mathscr{X}$ and $\mathsf{Y} \in \mathscr{Y}$ and any function $f:\mathscr{X} \to \mathbb{R}$, 
\begin{gather} \label{max_indicator}
    \max_{\mathsf{x} \in [[\mathsf{X}|\mathsf{y}]]} f(\mathsf{x}) = \max_{\mathsf{x} \in \mathscr{X}} \big(f(\mathsf{x}) + \mathbb{I}(\mathsf{x}|\mathsf{y}) \big), \quad \forall \mathsf{y} \in \mathscr{Y}.
\end{gather}
We also require the \textit{accrued distribution} for an uncertain variable at each $t$, defined using the accrued cost $A_t \in \mathcal{A}_t$.

\begin{definition} \label{def_accrued_cost}
Let $\mathsf{X} \in \mathscr{X}$ and $\mathsf{Y} \in \mathscr{Y}$ be two finite uncertain variables and let $A_t \in \mathcal{A}_t$ be the accrued cost at any $t=0,\dots,T$. An \textit{accrued distribution} at any $t$ for any $\mathsf{x} \in \mathscr{X}$ is a function $r_t: \mathscr{X} \to \{-\infty\} \cup [-a_t^{\max}, 0]$, given by
\begin{gather}
    r_t(\mathsf{x}) := \max_{a_t \in \mathcal{A}_t}\big(a_t + \mathbb{I}(\mathsf{x}, a_t)\big) - \max_{a_t \in \mathcal{A}_t} \big(a_t + \mathbb{I}(a_t) \big), 
\end{gather}
and for $\mathsf{x} \in \mathscr{X}$ given a realization $\mathsf{y} \in \mathscr{Y}$, $\mathbb{I}(\mathsf{y} ) > - \infty$, it is a function $r_t: \mathscr{X} \times \mathscr{Y} \to \{-\infty\} \cup [-a_t^{\max}, 0]$, given by
\begin{align}
    \hspace{-5pt} r_t(\mathsf{x}|\mathsf{y}) \hspace{-2pt} := \hspace{-2pt} \max_{a_t \in \mathcal{A}_t} \hspace{-2pt} \big(a_t + \mathbb{I}(\mathsf{x}, a_t|\mathsf{y}) \big) \hspace{-2pt}
    - \hspace{-2pt} \max_{a_t \in \mathcal{A}_t} \hspace{-2pt} \big(a_t + \mathbb{I}(a_t|\mathsf{y}) \big), \hspace{-4pt} \label{def_r}
\end{align}
where $a_t^{\max} := \max \mathcal{A}_t$.
\end{definition}

At each $t=0,\dots,T$, note that the accrued distribution $r_t(\mathsf{x}|\mathsf{y}) = -\infty$ if $\mathsf{x} \not\in [[\mathsf{X}|\mathsf{y}]]$ whereas $r_t(\mathsf{x}|\mathsf{y}) \in [-a_t^{\max},0]$ if $\mathsf{x} \in [[\mathsf{X}|\mathsf{y}]]$. It satisfies the properties to be a valid cost distribution. 
Furthermore, we can compute the conditional range $[[X_t, M_{t+1}|m_t, u_t]]$ at any $t$ given the realizations $m_t \in \mathcal{M}_t$ and $u_t \in \mathcal{U}$. Subsequently, we can use Definitions \ref{def_ind} - \ref{def_accrued_cost} to derive the accrued distribution $r_t(x_t, m_{t+1}|m_t, u_t)$, for all $x_t \in \mathcal{X}$ and $m_{t+1} \in \mathcal{M}_{t+1}$. Then,
we use it in the specialized DP decomposition for Problem \ref{problem_1} as follows. For all $m_t \in \mathcal{M}_t$ and $u_t \in \mathcal{U}$, for all $t=0,\dots,T-1$, we define
\begin{align}
    Q_t(m_t,u_t) \hspace{-2pt} := \hspace{-2pt}  &\max_{x_t \in \mathcal{X}, m_{t+1} \in \mathcal{M}_{t+1}} \big( c_t(x_t,u_t) + V_{t+1}(m_{t+1}) \nonumber \\
    &\quad \quad \quad \quad \quad \quad + r_t(x_t, m_{t+1}|m_t, u_t) \big), \label{DP_ad_1} \\
    V_t(m_t) \hspace{-2pt} := &\min_{u_t \in \mathcal{U}} Q_t(m_t, u_t), \label{DP_ad_2}
\end{align}
where, at time $T$, $Q_T(m_T,u_T) := \max_{x_T \in \mathcal{X}}$ $\big( c_T(x_T,u_T) + r_T(x_T|m_T)\big)$ and $V_T(m_T) := \min_{u_T \in \mathcal{U}} Q_T(m_T, u_T)$. We define the corresponding control law at time $t$ as $g_t^*(m_t) := \arg\min_{u_t \in \mathcal{U}} Q_t(m_t,u_t)$ and the control strategy as $\boldsymbol{g}^*=(g_0^*, \dots, g_T^*)$. Next, we show that solving the DP \eqref{DP_ad_1} - \eqref{DP_ad_2} computes the optimal performance and control strategy.

\begin{theorem} \label{thm_equivalence}
For all $m_t \in \mathcal{M}_t$ and $u_t \in \mathcal{U}$, for all $t=0,\dots,T$,
\begin{align}
    Q_t^{\text{tm}}(m_t, u_t) &= Q_t(m_t, u_t) + \max_{a_t \in [[A_t|m_t]]} a_t, \label{eq_thm_ad_m_1} \\
    V_t^{\text{tm}}(m_t) &= V_t(m_t) + \max_{a_t \in [[A_t|m_t]]} a_t, \label{eq_thm_ad_m_2}
\end{align}
and furthermore, $\boldsymbol{g}^*$ is an optimal solution to Problem \ref{problem_1}.
\end{theorem}

\begin{proof}
See Appendix A.
\end{proof}

Thoerem \ref{thm_equivalence} establishes that the specialized DP \eqref{DP_ad_1} - \eqref{DP_ad_2} computes an optimal solution to Problem \ref{problem_1}. Note that at each $t$, the optimization in the RHS of \eqref{DP_ad_2} must still be solved for each possible $m_t \in \mathcal{M}_t$, in a manner similar to \eqref{DP_ad_basic_1} - \eqref{DP_ad_basic_2} Thus, we still require a large number of computations for longer time horizons. In the next subsection, we define \textit{information states} to address this concern. 

\subsection{Information States} \label{subsection:info_state}

In this subsection, we introduce information states to construct an optimal DP decomposition for Problem \ref{problem_1}.

\begin{definition} \label{def_info_ad_state}
An \textit{information state} at any $t=0,\dots,T$ is an uncertain variable $\Pi_t= \sigma_t(M_t)$ taking values in a finite set $\mathcal{P}_t$, where $\sigma_t: \mathcal{M}_t \to \mathcal{P}_t$. Furthermore, for all $t$, for all $m_t \in \mathcal{M}_t$, $u_t \in \mathcal{U}$, $x_t \in \mathcal{X}$ and $\pi_{t+1} \in \mathcal{P}_{t+1}$, it satisfies:
\begin{align}
    r_t(x_t, \pi_{t+1}|m_t, u_t) &= r_t(x_t, \pi_{t+1}|\sigma_t(m_t), u_t), \nonumber \\ & \quad \quad \quad \quad \quad \quad \quad t =0,\dots, T-1, \label{p_ad_t}\\
    r_T(x_T|m_T) &= r_T(x_T|\sigma_t(m_T)). \label{p_ad_final}
\end{align}
\end{definition}

In the corresponding DP, for all $\pi_t \in \mathcal{P}_t$ and $u_t \in \mathcal{U}$, for all $t=0,\dots,T-1$, we define the value functions
\begin{align}
    \bar{Q}_t(\pi_t, u_t) := &\max_{x_t \in \mathcal{X}, \pi_{t+1} \in \mathcal{P}_{t+1}} \big ( \bar{V}_{t+1}(\pi_{t+1}) +c_t(x_t,u_t) \nonumber \\
    &\quad \quad \quad \quad \quad \quad  + r_t(x_t, \pi_{t+1}|\pi_t, u_t) \big ), \label{DP_info_ad_1}\\
    \bar{V}_t(\pi_t) := &\min_{u_t \in \mathcal{U}} \bar{Q}_t(\pi_t, u_t), \label{DP_info_ad_2}
\end{align}
where, at time $T$, $\bar{Q}_T(\pi_T, u_T) := \max_{x_T \in \mathcal{X}} \big(c_T(x_T,u_T) + r_T(x_T|\pi_T)\big)$ and $\bar{V}_T(\pi_T) := \min_{u_T \in \mathcal{U}} \bar{Q}_T(\pi_T, u_T)$. The control law at each $t$ is $\bar{g}_t^*(\pi_t) := \arg \min_{u_t \in \mathcal{U}} \bar{Q}_t(\pi_t, u_t)$. Next, we prove that the information state based DP \eqref{DP_info_ad_1} - \eqref{DP_info_ad_2} yields the same value as the specialized DP \eqref{DP_ad_1} - \eqref{DP_ad_2}.

\begin{theorem} \label{thm_opt_ad_dp}
Let $\Pi_t = \sigma_t(M_t)$ be an information state at each $t=0,\dots,T$. Then, for all $m_t \in \mathcal{M}_t$ and $u_t \in \mathcal{U}$, ${Q}_t(m_t, u_t) = \bar{Q}_t(\sigma_t(m_t), u_t)$ and $V_t(m_t) = \bar{V}_t(\sigma_t(m_t))$.
\end{theorem}

\begin{proof}
See Appendix B.
\end{proof}

From Theorem \ref{thm_opt_ad_dp}, the strategy $\bar{\boldsymbol{g}}^* = (\bar{g}_0^*, \dots, \bar{g}_T^*)$ using information states is an optimal solution to Problem \ref{problem_1}. In practice, using information states to compute $\bar{\boldsymbol{g}}^*$ is more tractable than using the memory to compute $\boldsymbol{g}^*$ only when the set $\mathcal{P}_t$ has fewer elements than $\mathcal{M}_t$ for most instances of $t$. This is usually true for systems with long time horizons.

\subsection{Examples of Information States} \label{subsection:info_examples}

In this subsection, we present examples of information states which satisfy the conditions in Definition \ref{def_info_ad_state}.


\textit{1) Partially observed systems:} Generally, at each $t=0,\dots,T$ a valid information state which satisfies Definition \ref{def_info_ad_state} is the function valued uncertain variable $\Pi_t: \mathcal{X} \to \{-\infty\} \cup [-a_t^{\max}, 0]$. At time $t$, for a given $m_t \in \mathcal{M}_t$, the realization of $\Pi_t$ is $p_t(x_t) := r_t(x_t|m_t) = \max_{a_t \in \mathcal{A}_t}\big(a_t + \mathbb{I}(x_t,a_t|m_t)\big) - \max_{a_t \in \mathcal{A}_t}\big(a_t + \mathbb{I}(a_t|m_t)\big)$ for all $x_t \in \mathcal{X}$. 
Note that this can be interpreted as a normalization \cite{bernhard1996separation} of the standard information state from \cite{james1994risk,  bernhard2003minimax}.

\textit{2) Perfectly observed systems:} Consider a system where $Y_t = X_t$ for all $t$. An information state for such a system is $\Pi_t = X_t$ at each $t$, i.e, the state itself. This information state is simpler than the one in Case 1.


\textit{3) Systems with action dependent costs:} Consider a partially observed system where at each $t$ the cost has the form $c_t(U_t) \in \mathbb{R}_{\geq0}$, and the terminal cost is $c_T(X_T, U_T)$. Then, an information state is the conditional range $\Pi_t = [[X_t|M_t]]$ at each $t$ (see Appendix C of our online preprint \cite{dave2022additive}). 

\begin{remark} \label{remark:improvement}
From \cite{Dave2022approx}, we know that the terminal DP \eqref{DP_ad_basic_1} - \eqref{DP_ad_basic_2} can be used to derive another information state $\Xi_t = [[X_t, A_t| M_t]]$ for each $t$ for Case 1. The conditional range $\Xi_t$ can take $2^{|\mathcal{A}_t|\times|\mathcal{X}|}$ feasible values whereas $\Pi_t$ from Case 1 can take $|\mathcal{A}_t|^{|\mathcal{X}|}$ values. As $|\mathcal{A}_t|$ grows in size with time $t$, the number of feasible values of $\Pi_t$ increases at a slower rate than the number of feasible values of $\Xi_t$. Thus, $\Pi_t$ yields a more computationally tractable DP than $\Xi_t$. This illustrates that constructing information states using the specialized DP \eqref{DP_ad_1} - \eqref{DP_ad_2} is better than using the terminal DP \eqref{DP_ad_basic_1} - \eqref{DP_ad_basic_2}. 
\end{remark}


\begin{remark}
Using Definition \ref{def_info_ad_state} we can identify simpler information states for systems with special properties, as shown in Cases 2 - 3. However, in many applications, merely using an information state may not sufficiently improve the tractbility optimal strategies. 
Thus, we extend Definition \ref{def_info_ad_state} to include approximate information states in Section \ref{section:approx}.
\end{remark}

\section{Approximate Information States} \label{section:approx}

In this section, we define approximate information states and utilize them to develop an approximate DP. We begin by defining a distance between two cost distributions.

\begin{definition} \label{def_cost_metric}
Let $\mathscr{X}$ be a finite subset of a metric space $(\mathscr{S}, d)$, with an uncertain variable $X \in \mathscr{X}$ and two distributions $r: \mathscr{X} \to \{-\infty\} $ $\cup[-a^1, 0]$ and $q: \mathscr{X} \to \{-\infty\} \cup [-a^2, 0]$, $a^1, a^2\in \mathbb{R}_{\geq0}$. Then:

1) The \textit{finite domains} of $r$ and $q$ are the sets $\mathscr{X}^r \hspace{-2pt} := \hspace{-2pt} \{\mathsf{x} \hspace{-1pt} \in \hspace{-1pt} \mathscr{X}$ $|r(\mathsf{x}) \hspace{-1pt} \neq \hspace{-1pt} -\infty\}$ and $\mathscr{X}^q \hspace{-2pt} := \hspace{-2pt} \{\mathsf{x} \hspace{-1pt} \in \hspace{-1pt} \mathscr{X}| q(\mathsf{x}) \hspace{-1pt} \neq \hspace{-1pt} -\infty\},$ respectively.

2) For any $\mathsf{x} \in \mathcal{X}^r \cup \mathcal{X}^q$, the \textit{nearest finite inputs} for $r$ and $q$ are given by $\psi^r(\mathsf{x}) := \arg\min_{\hat{\mathsf{x}} \in \mathscr{X}^r} d(\hat{\mathsf{x}},\mathsf{x}),$ and $\psi^q(\mathsf{x}) := \arg\min_{\hat{\mathsf{x}} \in \mathscr{X}^q}$ $d(\hat{\mathsf{x}},\mathsf{x})$, respectively.

3) The \textit{distance} between the distributions $r$ and $q$ is
\begin{multline} \label{eq_cost_metric}
    \mathcal{R}\big(r, q\big) := \max \big( \mathcal{H}(\mathscr{X}^r, \mathscr{X}^q), \\
    \max_{\mathsf{x} \in \mathscr{X}^r \cup \mathscr{X}^q} |{r}(\psi^r(\mathsf{x})) - {q}(\psi^q(\mathsf{x}))| \big),
\end{multline}
where $\mathcal{H}$ is the Hausdorff metric.
\end{definition}

\begin{remark}
Because any cost distribution cannot identically return $- \infty$ for all $\mathsf{x} \in \mathscr{X}$, the sets $\mathscr{X}^r$ and $\mathscr{X}^q$ are non-empty for all distributions $r, q$ on $\mathsf{X}$. Consequently, the distance $\mathcal{R}(r,q)$ always returns a finite value.
\end{remark}

Note that $\mathcal{R}$ is the maximum of a metric on a set-space and a metric on a function-space. Thus, it can quantify the distance between two different \textit{accrued distributions} on an uncertain variable $\mathsf{X} \in \mathscr{X}$. Specifically, let $\mathsf{Y} \in \mathscr{Y}$ and $\mathsf{Z} \in \mathscr{Z}$ take realizations $y \in \mathcal{Y}$ and $z \in \mathcal{Z}$, respectively, such that $[[\mathsf{X}, A_t|\mathsf{y}]] \neq \emptyset$ and $[[\mathsf{X}, A_t|\mathsf{z}]] \neq \emptyset$ for some time $t$. Then, we denote the functional forms of the conditional distributions on $\mathsf{X}$ given $y$ and given $z$ as $r_t(\mathsf{X}|\mathsf{y})$ and $r_t(\mathsf{X}|\mathsf{z})$, respectively, 
and quantify the distance between them as
\begin{multline} \label{eq_distribution_distance}
    \mathcal{R}\big(r_t(\mathsf{X}|\mathsf{y}), r_t(\mathsf{X}|\mathsf{z})\big) := \max\big( \mathcal{H}\big([[\mathsf{X}|\mathsf{y}]], [[\mathsf{X}|\mathsf{z}]]\big), \\
    \max_{\mathsf{x} \in [[\mathsf{X}|\mathsf{y}]] \cup [[\mathsf{X}|\mathsf{z}]]} \big|r_t\big(\psi^{\mathsf{y}}(\mathsf{x})|\mathsf{y}\big) - r_t\big(\psi^{\mathsf{z}}(\mathsf{x})|\mathsf{z}\big) \big| \big),
\end{multline}
where, the finite domains are $\{\mathsf{x} \in \mathscr{X}~|~r_t(\mathsf{x}|\mathsf{y}) \neq -\infty\} = [[\mathsf{X}|\mathsf{y}]]$ and $\{\mathsf{x} \in \mathscr{X}~|~r_t(\mathsf{x}|\mathsf{z}) \neq -\infty\} = [[\mathsf{X}|\mathsf{z}]]$; and for any $\mathsf{x} \in [[\mathsf{X}|\mathsf{y}]] \cup [[\mathsf{X}|\mathsf{z}]]$, the nearest finite inputs are $\psi^{\mathsf{y}}(\mathsf{x}) :=  \arg\min_{\hat{\mathsf{x}} \in [[\mathsf{X}|\mathsf{y}]]} d(\hat{\mathsf{x}},\mathsf{x})$ and $\psi^{\mathsf{z}}(\mathsf{x}) :=  \arg\min_{\hat{\mathsf{x}} \in [[\mathsf{X}|\mathsf{z}]]} d(\hat{\mathsf{x}}, \mathsf{x})$. 
Next, using $\mathcal{R}$ to quantify the approximation gap, we define approximate information states for Problem \ref{problem_1}.


\begin{definition} \label{def_ad_approx}
An \textit{approximate information state} at any $t=0,\dots,T$ is an uncertain variable $\hat{\Pi}_t = \hat{\sigma}_t(M_t)$ taking values in a finite subset $\hat{\mathcal{P}}_t$ of some metric space, where $\hat{\sigma}_t : \mathcal{M}_t \to \hat{\mathcal{P}}_t$. Furthermore, for all $t$, there exists a parameter $\epsilon_t \in \mathbb{R}_{\geq0}$ such that for all $m_t \in \mathcal{M}_t$ and $u_t \in \mathcal{U}$, it satisfies:
\begin{align}
    &\mathcal{R}\big(r_t(X_t, \hat{\Pi}_{t+1}~|~m_t,u_t), r_t(X_t, \hat{\Pi}_{t+1}~|~\hat{\sigma}_t(m_t),u_t)\big) \nonumber \\
    &\quad \quad \quad \quad \quad \quad \quad \quad \quad \leq \epsilon_t, \quad t =0,\dots,T-1, \label{ap_ad_1}\\
    &\mathcal{R}\big(r_T(X_T~|~m_T), r_T(X_T~|~\hat{\sigma}_T(m_T))\big) \leq \epsilon_T. \label{ap_ad_2}
\end{align}
\end{definition}

In the approximate DP, for all $t=0,\dots,T-1$, for all $\hat{\pi}_t \in \hat{\mathcal{P}}_t$ and $u_t \in \mathcal{U}$, we recursively define the value functions
\begin{align}
    \hat{Q}_t(\hat{\pi}_t, u_t) := &\max_{x_t \in \mathcal{X}, \hat{\pi}_{t+1} \in \hat{\mathcal{P}}_{t+1}} \big ( \hat{V}_{t+1}(\hat{\pi}_{t+1}) +c_t(x_t,u_t) \nonumber \\
    & \quad \quad \quad \quad \quad \quad + r_t(x_t, \hat{\pi}_{t+1}|\hat{\pi}_t, u_t) \big ), \label{DP_approx_ad_1}\\
    \hat{V}_t(\hat{\pi}_t) := &\min_{u_t \in \mathcal{U}} \hat{Q}_t(\hat{\pi}_t, u_t), \label{DP_approx_ad_2}
\end{align}
where, at time $T$, $\hat{Q}_T(\hat{\pi}_T, u_T) := \max_{x_T \in \mathcal{X}} \big(c_T(x_T,u_T) + r_T(x_T|\hat{\pi}_T, u_T)\big)$ and $\hat{V}_T(\hat{\pi}_T) := \min_{u_T \in \mathcal{U}} \hat{Q}_T(\hat{\pi}_T, u_T)$. The control law at each $t$ is $\hat{g}_t^*(\hat{\pi}_t) := \arg \min_{u_t \in \mathcal{U}} \hat{Q}_t(\hat{\pi}_t, u_t)$ and the approximate control strategy is $\hat{\boldsymbol{g}}^* := (\hat{g}_0^*, \dots, \hat{g}_T^*)$. 
Next, we bound the performance loss from implementing the approximate control strategy $\hat{\boldsymbol{g}}^*$ in Problem \ref{problem_1}. We begin with a preliminary result which will be required subsequently.

\begin{lemma} \label{lem_ad_prelim}
Let $\mathscr{X}$ be a finite subset of a metric space $(\mathscr{S}, d)$ and consider two cost distributions $r: \mathscr{X} \to \{-\infty\} \cup [-a^1, 0]$ and $q: \mathscr{X} \to \{-\infty\} \cup [-a^2, 0]$, where $a^1, a^2\in \mathbb{R}_{\geq0}$. Then, for a Lipschitz function $f: \mathscr{X} \to \mathbb{R}$:
\begin{multline} \label{eq_ad_prelim}
         \big| \max_{\mathsf{x} \in \mathscr{X}} \big( f(\mathsf{x})  + r(\mathsf{x})\big) - \max_{\mathsf{x} \in \mathscr{X}}  \big(  f(\mathsf{x})  +  q(\mathsf{x})  \big)  \big| \\
        \leq (L_{f} + 1) \cdot \mathcal{R}\big(r, q \big).
    \end{multline}
\end{lemma}

\begin{proof}
See Appendix D.
\end{proof}

Next, we bound the maximum error when approximating the value functions in the optimal DP \eqref{DP_ad_basic_1} - \eqref{DP_ad_basic_2} with the value functions in the approximate DP \eqref{DP_approx_ad_1} - \eqref{DP_approx_ad_2}.

\begin{theorem} \label{thm_approx_ad_dp}

Let $L_{\hat{V}_{t+1}}$ be the Lipschitz constant of $\hat{V}_{t+1}$ for all $t=0,\dots,T-1$. Then, for all $m_t \in \mathcal{M}_t$ and $u_t \in \mathcal{U}$,
\begin{align}
    |Q_t(m_t,u_t) - \hat{Q}_t(\hat{\sigma}_t(m_t),u_t)| &\leq \alpha_t, \label{thm_7_1} \\
    |V_t(m_t) - \hat{V}_t(\hat{\sigma}_t(m_t))| &\leq \alpha_t, \label{thm_7_2}
\end{align}
where $\alpha_t = \alpha_{t+1} + ( 2 L_t + 1 )\cdot \epsilon_t$, where $L_t := \max\{L_{\hat{V}_{t+1}},$ $L_{c_t}\}$, for all $t=0,\dots,T-1$ and $\alpha_{T} = (L_{c_T} + 1) \cdot \epsilon_T$.
\end{theorem}

\begin{proof}
See Appendix E.
\end{proof}

Next, we bound the maximum difference in the performance of an approximate control strategy $\boldsymbol{\hat{g}}^* := (\hat{g}_0^*, \dots, \hat{g}_T^*)$ and optimal strategy $\boldsymbol{g}^*$. Recall that  $\hat{g}_t^*(\hat{\pi}_t) = \arg \min_{u_t \in \mathcal{U}} \hat{Q}_t(\hat{\pi}_t, u_t)$ for all $t=0,\dots,T$. Then, the equivalent strategy $\boldsymbol{g}=(g_0,\dots,g_T)$, which utilizes the memory but yield the same actions and performance as $\boldsymbol{\hat{g}}^*$, is constructed as $g_t(m_t) := \hat{g}_t^*(\hat{\sigma}_t(m_t))$ for all $t$. To compute the performance of $\boldsymbol{g}$ (and consequently, of $\boldsymbol{\hat{g}}^*$), we define for all $t=0,\dots,T-1$, for all $m_t \in \mathcal{M}_t$ and $u_t \in \mathcal{U}$,
\begin{align}
\Theta_t(m_t,u_t) :=  &\max_{x_t \in \mathcal{X}, m_{t+1} \in \mathcal{M}_{t+1}} \big(\Lambda_{t+1}(m_{t+1}) + c_t(x_t,u_t) \nonumber \\
&\quad \quad \quad \quad \quad \quad  + r_t(x_t, m_{t+1}|m_t, u_t) \big), \label{value_g_ad_1}\\
\Lambda_t(m_t) := &\Theta_t(m_t,g_t(m_t)), \label{value_g_ad_2} 
\end{align}
where, at time $T$, $\Theta_{T}(m_{T}, u_T) := \max_{x_T \in \mathcal{X}}(c_T(x_T,u_T) + r_T(x_T|m_T,u_T))$ and $V_T(m_T) = \Theta_T(m_T, g_T(m_T))$. Recursively evaluating the value functions \eqref{value_g_ad_1} - \eqref{value_g_ad_2} computes the performance of $\boldsymbol{g}$ as $\Lambda_0(m_0)$, where $m_0 = y_0$. Note that the performance of $\boldsymbol{g}^*$ is simply the optimal value. Next, we bound the difference in the performances of $\boldsymbol{g}$ and $\boldsymbol{g}^*$.

\begin{theorem} \label{thm_approx_ad_policy}
Let $L_{\hat{V}_{t+1}}$ be the Lipschitz constant of $\hat{V}_{t+1}$ for all $t=0,\dots,T-1$. Then, for all $m_t \in \mathcal{M}_t$ and $u_t \in \mathcal{U}$,
\begin{align}
    |Q_t(m_t, u_t) - \Theta_t(m_t, u_t)| \leq 2\alpha_t, \label{thm_8_3} \\
    |V_t(m_t) - \Lambda_t(m_t)| \leq 2\alpha_t. \label{thm_8_4}
\end{align}
where $\alpha_t = \alpha_{t+1} + ( 2 L_t + 1 )\cdot \epsilon_t$ with $L_t := \max\{L_{\hat{V}_{t+1}},$ $L_{c_t}\}$ for all $t=0,\dots,T-1$ and $\alpha_{T} = (L_{c_T} + 1) \cdot \epsilon_T$.
\end{theorem}

\begin{proof}
See Appendix F.
\end{proof}

\section{Numerical Example} \label{section:example}

For our numerical example, we consider an agent pursuing a target across a $9 \times 9$ grid with obstacles. At each $t=0,\dots,T$, the agent's position is $X_t^\text{ag}$ and the target's position is $X_t^{\text{ta}}$, each of which takes values in the set of grid cells $\mathcal{X} = \big\{(-4,-4),(-4,-3),\dots,(3,4),(4,4)\big\} \setminus \mathcal{O}$, where $\mathcal{O} \subset \mathcal{X}$ is a known set of obstacle cells. Let $\mathcal{W} = \mathcal{N} =  \{(-1,0),(1,0),(0,0),(0,1),$ $(0,-1)\}$ and $\mathcal{D} := \{(-1,1), (1,1), (1,-1), (-1,-1) \}$. Starting at $X_0^{\text{ta}} \in \mathcal{X}$, the target's position evolves as $X_{t+1}^{\text{ta}} = \delta(X_t^{\text{ta}} + W_t \in \mathcal{X})\cdot(X_t^{\text{ta}} + W_t) + (1 - \delta(X_t^{\text{ta}} + W_t \in \mathcal{X}) )\cdot X_t^{\text{ta}}$, where $W_t \in \mathcal{W}$ and $\delta(\cdot)$ returns $1$ if the condition in the argument holds and $0$ otherwise. 
At each $t$, the agent observes their own position perfectly and the target's position as $Y_t = \delta(X_t^{\text{ta}} + N_t \in \mathcal{X})\cdot(X_t^{\text{ta}} + N_t) + (1-\delta(X_t^{\text{ta}} + N_t \in \mathcal{X}) )\cdot X_t^{\text{ta}}$, where $N_t \in \mathcal{N}$. 
Then, the agent selects an action $U_t \in \mathcal{U} = \mathcal{W} \cup \mathcal{D}$ and moves as $X_{t+1}^{\text{ag}} = \delta(X_t^{\text{ag}} + U_t \in \mathcal{X})\cdot(X_t^{\text{ag}} + U_t) + (1 - \delta(X_t^{\text{ag}} + U_t \in \mathcal{X}) )\cdot X_t^{\text{ag}}$. 
The agent incurs an interim cost $c_t(U_t) := 0.5 \cdot \delta(U_t \in \mathcal{D})$ only if it moves diagonally, and a terimal cost $d(X_T^{\text{ta}}, X_T^{\text{ag}})$ corresponding to the final distance from the target.
We illustrate this in Fig. \ref{fig:1a}, where: (1) the black cells are obstacles, (2) the black triangle is the initial position of the agent and the hatched region around it indicates the available actions, and (3) the black circle is the initial observation of the agent and the hatched region around it indicates the possible initial positions of the target.

\begin{figure}[t!]
  \centering
  \subfigure[The original grid \label{fig:1a}]{\includegraphics[width=0.35\linewidth, keepaspectratio]{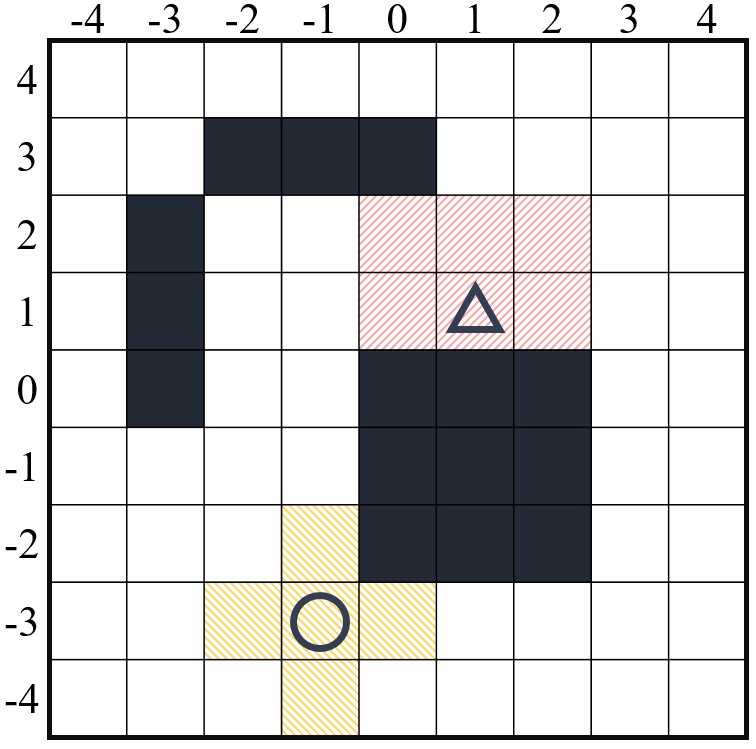}} 
  \hspace{15pt}
  \subfigure[The quantized grid \label{fig:1b}]{\includegraphics[width=0.35\linewidth, keepaspectratio]{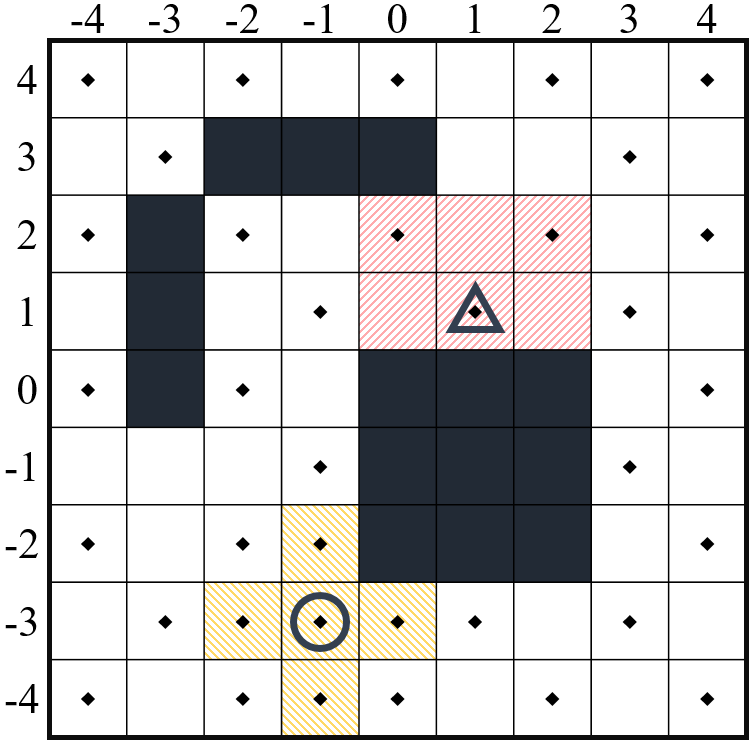}}
  \caption{The gridworld pursuit problem with the initial conditions $x_0^{\text{ag}} = (1,1)$ and $y_0 = (-1,-3)$.}
  \label{fig:illustration}
  \vspace{-12pt}
\end{figure}

This formulation is a system with action dependent costs as described in Subsection \ref{subsection:info_examples}. For such a system, an information state at time $t$ is $\Pi_t = (X_t^{\text{ag}}, \Lambda_t)$, where $\Lambda_t = [[X_t^{\text{ta}}|M_t]]$. We approximate $\Lambda_t$ at each $t$ using state quantization. First, we define a static set of quantized states $\hat{\mathcal{X}}$ such that $\max_{x_t \in \mathcal{X}} \min_{\hat{x_t} \in \hat{\mathcal{X}}} d(x_t,$ $ \hat{x}_t) \leq 1$ and a quantization function $\mu(x_t) := \arg \min_{\hat{x}_t \in \hat{\mathcal{X}}} d(x_t, \hat{x}_t)$ using the initial observation of the agent, as illustrated using dots in Fig. \ref{fig:1b}. Note that we use a finer quantization around the point of initial observation and sparser quantization elsewhere. 
Then, the approximate range at time $t$ is $\hat{\Lambda}_t  = \{ \mu(x_t) \in \hat{\mathcal{X}} ~|~ x_t \in \Lambda_t \}$ and the approximate information state is $\hat{\Pi}_t = \big(X_t^{\text{ag}}, \hat{\Lambda}_t, Y_0\big)$. We include $Y_0$ in $\hat{\Pi}_t$ because it facilitates the update of $\hat{\Lambda}_t$ to $\hat{\Lambda}_{t+1}$. For six initial conditions, we computed the best control strategy using both the optimal DP and approximate DP for $T=6$. In Fig. \ref{fig:table}, we have tabulated the worst-case values ($V_0$ and $\hat{V}_0$) and run-times in seconds (Run.) for both DPs. 
We also evaluated the difference between \textit{actual} costs incurred by the approximate strategy and the optimal strategy, respectively, by implementing both of them in 5000 simulations with randomly generated disturbances. We have marked these differences in Fig. \ref{fig:table} and indicated the frequency of each cost difference by the size of the disc marking it. While the approximate strategy is faster to compute than the optimal strategy for all cases, we note that it admits bounded deviations in actual costs. 

\begin{figure}[ht!]
  \centering
  \includegraphics[width=\linewidth, keepaspectratio]{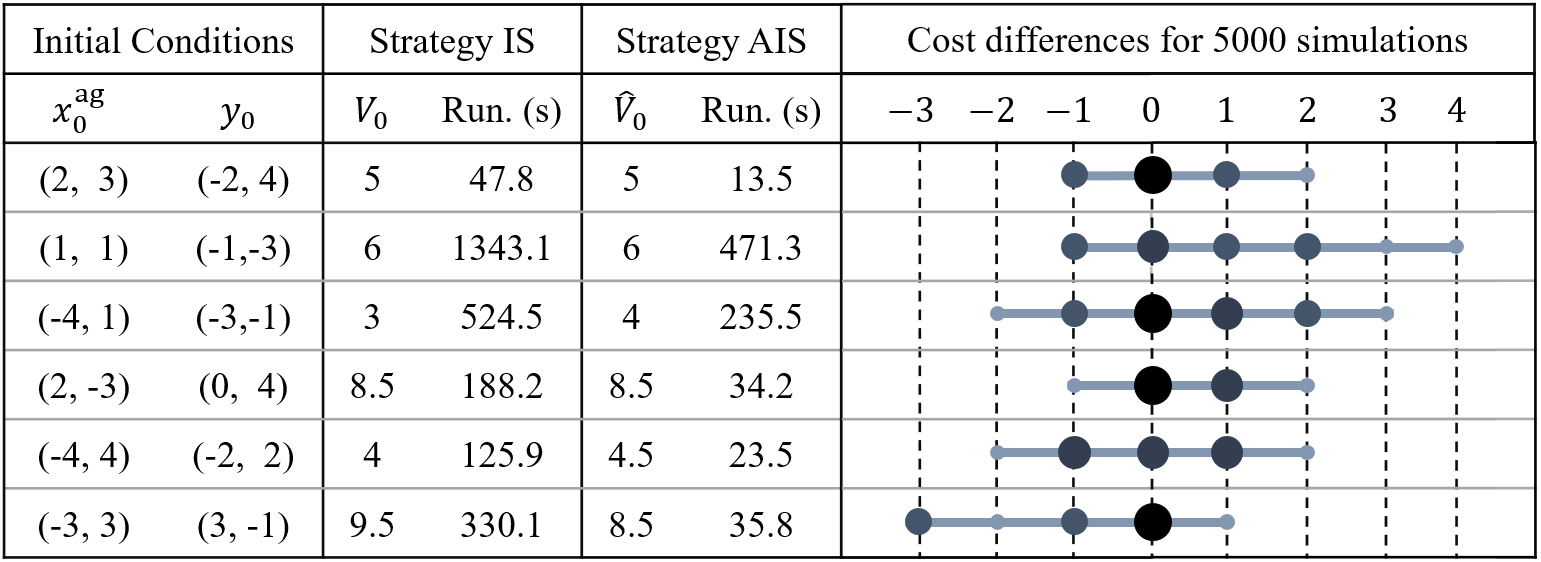}
  \vspace{-10pt}
  \caption{Results of numerical simulations for $T=6$.}
  \label{fig:table}
  \vspace{-12pt}
\end{figure}

\section{Conclusion} \label{section:conclusion}

In this paper, we developed a general theory of information states and approximate information states to tractably compute control strategies in non-stochastic additive cost problems. We used the theoretical framework of cost distributions to present a general definition for information states that compute an optimal control strategy. We showed that specific information states proposed in previous research efforts emerge as special cases of our definition. Then, we extended this definition to approximate information states which can be used to compute approximate control strategies which admit a bounded worst-case performance loss. Finally, using a numerical simulation, we illustrated the trade-off between computational tractability and performance loss inherent in the application of approximate information states. Future work should consider the use of this theory in non-stochastic reinforcement learning problems.

\bibliographystyle{ieeetr}
\bibliography{References,Latest_IDS}

\begin{thebibliography}{10}

\bibitem{Malikopoulos2021}
A.~A. Malikopoulos, ``On team decision problems with nonclassical information
  structures,'' {\em IEEE Transactions on Automatic Control}, 2023.

\bibitem{subramanian2022approximate}
J.~Subramanian, A.~Sinha, R.~Seraj, and A.~Mahajan, ``Approximate information
  state for approximate planning and reinforcement learning in partially
  observed systems,'' {\em Journal of Machine Learning Research}, vol.~23,
  no.~12, pp.~1--83, 2022.

\bibitem{Dave2021nestedaccess}
A.~Dave, N.~Venkatesh, and A.~A. Malikopoulos, ``On decentralized control of
  two agents with nested accessible information,'' in {\em 2022 American
  Control Conference (ACC)}, pp.~3423--3430, IEEE, 2022.

\bibitem{rasouli2018scalable}
M.~Rasouli, E.~Miehling, and D.~Teneketzis, ``A scalable decomposition method
  for the dynamic defense of cyber networks,'' in {\em Game Theory for Security
  and Risk Management}, pp.~75--98, Springer, 2018.

\bibitem{giuliani2021state}
M.~Giuliani, J.~Lamontagne, P.~Reed, and A.~Castelletti, ``A state-of-the-art
  review of optimal reservoir control for managing conflicting demands in a
  changing world,'' {\em Water Resources Research}, vol.~57, no.~12,
  p.~e2021WR029927, 2021.

\bibitem{mannor2007bias}
S.~Mannor, D.~Simester, P.~Sun, and J.~N. Tsitsiklis, ``Bias and variance
  approximation in value function estimates,'' {\em Management Science},
  vol.~53, no.~2, pp.~308--322, 2007.

\bibitem{james1994risk}
M.~R. James, J.~S. Baras, and R.~J. Elliott, ``Risk-sensitive control and
  dynamic games for partially observed discrete-time nonlinear systems,'' {\em
  IEEE transactions on automatic control}, vol.~39, no.~4, pp.~780--792, 1994.

\bibitem{bernhard2003minimax}
P.~Bernhard, ``Minimax - or feared value - $\text{$L$}_1$ / $\text{$L$}_\infty$
  control,'' {\em Theoretical computer science}, vol.~293, no.~1, pp.~25--44,
  2003.

\bibitem{bernhard1996separation}
P.~Bernhard, ``A separation theorem for expected value and feared value
  discrete time control,'' {\em ESAIM: Control, Optimisation and Calculus of
  Variations}, vol.~1, pp.~191--206, 1996.

\bibitem{nair2013nonstochastic}
G.~N. Nair, ``A nonstochastic information theory for communication and state
  estimation,'' {\em IEEE Transactions on automatic control}, vol.~58, no.~6,
  pp.~1497--1510, 2013.

\bibitem{rangi2019towards}
A.~Rangi and M.~Franceschetti, ``Towards a non-stochastic information theory,''
  in {\em 2019 IEEE International Symposium on Information Theory (ISIT)},
  pp.~997--1001, IEEE, 2019.

\bibitem{gradu2020non}
P.~Gradu, J.~Hallman, and E.~Hazan, ``Non-stochastic control with bandit
  feedback,'' {\em Advances in Neural Information Processing Systems}, vol.~33,
  pp.~10764--10774, 2020.

\bibitem{Dave2023approximate}
A.~Dave, N.~Venkatesh, and A.~A. Malikopoulos, ``{Approximate Information
  States for Worst-Case Control and Learning in Uncertain Systems},'' {\em
  arXiv:2301.05089 (in review)}, 2023.

\bibitem{gagrani2017decentralized}
M.~Gagrani and A.~Nayyar, ``Decentralized minimax control problems with partial
  history sharing,'' in {\em 2017 American Control Conference (ACC)},
  pp.~3373--3379, IEEE, 2017.

\bibitem{Dave2021minimax}
A.~Dave, N.~Venkatesh, and A.~A. Malikopoulos, ``On decentralized minimax
  control with nested subsystems,'' in {\em 2022 American Control Conference
  (ACC)}, pp.~3437--3444, IEEE, 2022.

\bibitem{bertsekas1973sufficiently}
D.~Bertsekas and I.~Rhodes, ``Sufficiently informative functions and the
  minimax feedback control of uncertain dynamic systems,'' {\em IEEE
  Transactions on Automatic Control}, vol.~18, no.~2, pp.~117--124, 1973.

\bibitem{bacsar2008h}
T.~Ba{\c{s}}ar and P.~Bernhard, {\em H-infinity optimal control and related
  minimax design problems: a dynamic game approach}.
\newblock Springer Science \& Business Media, 2008.

\bibitem{wiesemann2013robust}
W.~Wiesemann, D.~Kuhn, and B.~Rustem, ``Robust markov decision processes,''
  {\em Mathematics of Operations Research}, vol.~38, no.~1, pp.~153--183, 2013.

\bibitem{rasouli2018robust}
M.~Rasouli and S.~Saghafian, ``Robust partially observable markov decision
  processes,'' 2018.

\bibitem{Dave2022approx}
A.~Dave, N.~Venkatesh, and A.~A. Malikopoulos, ``Approximate information states
  for worst-case control of uncertain systems,'' in {\em Proceedings of the
  61th IEEE Conference on Decision and Control (CDC)}, pp.~4945--4950, 2022.

\bibitem{barnsley2006superfractals}
M.~F. Barnsley, {\em Superfractals}.
\newblock Cambridge University Press, 2006.

\bibitem{kolokoltsov1997idempotent}
V.~Kolokoltsov and V.~P. Maslov, {\em Idempotent analysis and its
  applications}, vol.~401.
\newblock Springer Science \& Business Media, 1997.

\bibitem{dave2022additive}
A.~Dave, N.~Venkatesh, and A.~A. Malikopoulos, ``On non-stochastic approximate
  information states for uncertain systems with additive costs,'' {\em arXiv
  preprint, arXiv:2209.13787}, 2022.

\end{thebibliography}

\section*{Appendix A - Proof of Theorem 1}

For all $t=0,\dots,T$, let $m_t \in \mathcal{M}_t$ and $u_t \in \mathcal{U}$ be given realizations of $M_t$ and $U_t$, respectively. We prove the result by mathematical induction starting at the last time step $T$. We use \eqref{max_indicator} to expand the left hand side (LHS) of \eqref{eq_thm_ad_m_1} as
${Q}_T^{\text{tm}}(m_T,u_T) = \max_{a_T, x_T \in [[A_T, X_T|m_T, u_T]]}$ $\big(c_T(x_T,u_T\big) + a_T)$
    $= \max_{a_T \in \mathcal{A}_T, x_T \in \mathcal{X}} \big( x_T + a_T + \mathbb{I}(a_T,x_T|m_T, u_T) \big).$
In the RHS, we add and subtract $\max_{a_T \in \mathcal{A}_T}(a_T + \mathbb{I}(a_T|m_T))$ to write that
${Q}_T^{\text{tm}}(m_T,u_T)  =  \max_{x_T \in \mathcal{X}}   \big( c_T(x_T,u_T) +  \max_{a_T \in \mathcal{A}_T}  \big(a_T + \mathbb{I}(a_T, x_T$
    $|m_T)\big) -  \max_{a_T \in \mathcal{A}_T}  \big( a_T \hspace{-2pt} +  \mathbb{I}(a_T|m_T)\big) \big)   +  \max_{a_T \in \mathcal{A}_T}  \big( a_T  +  \mathbb{I}(a_T |m_T) \big)$ 
   $=  \max_{x_T \in \mathcal{X}}  \big( c_T(x_T,u_T)  +  r_T(x_T|m_T)\big)  +  \max_{a_T \in [[A_T|m_T]]}  a_T,$
where, in the second equality, we use \eqref{def_r} from Definition \ref{def_accrued_cost}. Thus, using the definition of $Q_T(m_T, u_T)$, we complete the proof for \eqref{eq_thm_ad_m_1} at time $T$. We can prove \eqref{eq_thm_ad_m_2} at time $T$ directly by minimizing both sides of \eqref{eq_thm_ad_m_1} with respect to $u_T \in \mathcal{U}$. Furthermore, note that $g_T(m_T) = \arg\inf_{u_T \in \mathcal{U}}{Q}_T(m_T,u_T) = \arg\inf_{u_T \in \mathcal{U}}$ ${Q}_T^{\text{tm}}(m_T,u_T)$, i.e., $u_T = g_T^*(m_T)$ minimizes ${Q}_T^{\text{tm}}(m_T,u_T)$. This forms the basis of our induction. Next, for all $t=0,\dots,T-1$, we consider the induction hypothesis $V_{t+1}^{\text{tm}}(m_{t+1}) = V_{t+1}(m_{t+1}) + \max_{a_{t+1} \in [[A_{t+1}|m_{t+1}]]} a_{t+1}$. 
Then, using the hypothesis $Q_{t}^{\text{tm}}(m_t,u_t) = \max_{m_{t+1} \in [[M_{t+1}|m_t,u_t]]} {V}_{t+1}^{\text{tm}}(m_{t+1}) = \max_{m_{t+1} \in [[M_{t+1}|m_t,u_t]]}$ $\big( V_t(m_{t+1})  + \max_{a_{t+1} \in [[A_{t+1}|m_{t+1}]]}$ $a_{t+1} \big)$
    $= \max_{m_{t+1}, a_{t+1} \in [[M_{t+1},A_{t+1}|m_t,u_t]]} 
    \big(V_t(m_{t+1}) + a_{t+1} \big) = \max_{m_{t+1}, x_{t}, a_{t} \in [[M_{t+1}, X_{t}, A_{t}|m_t,u_t]]}\big(V_t (m_{t+1})$
    $ + c_{t} (x_t,u_t) + a_{t} )$
    $= \max_{m_{t+1} \in \mathcal{M}_{t+1}, x_{t} \in \mathcal{X}_{t+1}, a_{t} \in \mathcal{A}_{t+1}}$
    $ \big(V_t(m_{t+1}) + c_{t}(x_t,u_t)$
    $+ a_{t} + \mathbb{I}(x_t, m_{t+1}, a_t|m_t, u_t) \big),$
where, in the third equality, we use the fact that $[[A_{t+1}|m_{t+1}]] = [[A_{t+1}|m_{t+1}, m_t, u_t]]$ because $m_{t+1} = (m_t, u_t, y_{t+1})$; in the fourth equality, we use the definition of $a_{t+1}$; and in the fifth equality, we use the property of the of the indicator function. Then, as for time $T$, we add and subtract $\max_{a_t \in \mathcal{A}_t}(a_t+\mathbb{I}(a_t|m_t))$ in the RHS and use \eqref{def_r} from Defintion \ref{def_accrued_cost} to conclude that
$Q_{t}^{\text{tm}}(m_t,u_t) = \max_{m_{t+1} \in \mathcal{M}_{t+1}, x_{t} \in \mathcal{X}_{t+1}} \big(V_t(m_{t+1}) + c_{t}(x_t,u_t)$
    $+ r_t(x_t, m_{t+1}|m_t,u_t) \big) -
    \max_{a_t \in \mathcal{A}_t}\big(a_{t} + \mathbb{I}(a_t|m_t) \big)$
    $= Q_t(m_t,u_t) + \max_{a_t \in [[A_t|m_t]]} a_t,$
which proves \eqref{eq_thm_ad_m_1} at time $t$. We can prove \eqref{eq_thm_ad_m_2} at time $t$ directly by minimizing both sides of \eqref{eq_thm_ad_m_1} respect to $u_t \in \mathcal{U}$, and furthermore, $g_t^*(m_t)$ $ = \arg\inf_{u_t \in \mathcal{U}}{Q}_t(m_t,u_t) = \arg\inf_{u_t \in \mathcal{U}}{Q}_t^{\text{tm}}(m_t,u_t)$. This proves the induction hypothesis at time $t$ and thus, the result holds for all $t=0,\dots,T$ using mathematical induction.

\section*{Appendix B - Proof of Theorem 2}

Let $m_t \in \mathcal{M}_t$ and $u_t \in \mathcal{U}$ be given realizations of $M_t$ and $U_t$, respectively, for all $t = 0,\dots,T$. We prove the result using mathematical induction starting with $T$, where 
$Q_{T}(m_{T}, u_T) = \max_{x_T \in \mathcal{X}} \big(c_T(x_T, u_T) + r_T(x_T|m_T)\big) = \max_{x_T \in \mathcal{X}} \big(c_T(x_T, u_T) + r_T(x_T|\sigma_t(m_T))\big) = \bar{Q}_{T}(\sigma_{T}(m_{T}), u_T)$ 
holds as a direct consequence of \eqref{p_ad_final} in Definition \ref{def_info_ad_state}. Subsequently, by taking the minimum on both sides with respect to $u_t \in \mathcal{U}$, it holds that $V_T(m_T) = \bar{V}_T(\sigma_T(m_T))$. 
With this as the basis, for each $t=0,\dots,T-1$, we consider the induction hypothesis $V_{t+1}(m_{t+1}) = \bar{V}_{t+1}(\sigma_{t+1}(m_{t+1}))$. Next, we prove that ${Q}_t(m_t, u_t) = \bar{Q}_t(\sigma_t(m_t), u_t)$ at time $t$ by showing that the RHS of \eqref{DP_ad_1} is equal to the RHS of \eqref{DP_info_ad_1}. Using the induction hypothesis in the RHS of \eqref{DP_ad_1}, $Q_t(m_t, u_t) = \max_{x_t \in \mathcal{X}, m_{t+1} \in \mathcal{M}_{t+1}} \big( V_{t+1}(m_{t+1}) + c_t(x_t, u_t)
    + r_t(x_t, m_{t+1}|m_t,u_t)\big) 
    = \max_{x_t \in \mathcal{X}, m_{t+1} \in \mathcal{M}_{t+1}}$ $\big( \bar{V}_{t+1}(\sigma_{t+1}(m_{t+1})) + c_t(x_t, u_t) 
    + r_t(x_t, m_{t+1}|m_t,u_t)\big)
    = \max_{x_t \in \mathcal{X}, \pi_{t+1} \in \mathcal{P}_{t+1}} \big( \bar{V}_{t+1}(\pi_{t+1}) + c_t(x_t, u_t)  
    + r_t(x_t, \pi_{t+1}|\pi_t,u_t) \big),$
where, in the second equality, we use result 2 from Lemma \ref{lem_cdist_property} and \eqref{p_ad_t}. Thus, at time $t$, it holds that $Q_t(m_t, u_t) = \bar{Q}_t(\sigma_t(m_t), u_t)$. Subsequently, we can prove $V_t(m_t) = \bar{V}_t(\sigma_t(m_t))$ by minimizing both sides with $u_t \in \mathcal{U}$. This proves the induction hypothesis at time $t$, and the result follows by mathematical induction.

\section*{Appendix C - Derivation of Information State for Systems with Action Dependent Costs}

In this appendix, we derive the information states for partially observed systems with control dependent costs as described in Subsection \ref{subsection:info_examples}. We recall that for a general partially observed system, the information state at each $t=0,\dots,T$ is given by the function $\Pi_t: \mathcal{X}_t \to \{-\infty\} \cup [-a_t^{\max},0]$. Given a realization $m_t \in \mathcal{M}_t$ of the memory $M_t$ at any time $t$, it takes as its realization the functional form $p_t(X_t) = r_t(X_t|m_t)$. Next, we prove an important result to establish the information state.

\begin{lemma} \label{lem_ad_case_4}
Let the incurred cost at each $t=0,\dots,T-1$ be $c_t(U_t) \in \mathbb{R}_{\geq0}$. Then, for any $m_t\in \mathcal{M}_t$ and $x_t \in \mathcal{X}_t$, it holds that $r_t(x_t|m_t) = \mathbb{I}(x_t|m_t)$.
\end{lemma}

\begin{proof}
Let $m_t \in \mathcal{M}_t$ and $x_t \in \mathcal{X}_t$ be realizations of the uncertain variables $M_t$ and $X_t$, respectively, at each $t=0,\dots,T$. Let $m_t = (y_{0:t}, u_{0:t-1})$ at time $t$. Then, we note that there exists a known function $\bar{c}_t: \prod_{\ell=0}^{t-1}\mathcal{U}_{\ell} \to \mathcal{A}_t$ such that $a_t = \bar{c}_t(u_{0:t-1})$. We use this property to write that
\begin{align}
    \hspace{-5pt} r_t(x_t|m_t) \hspace{-3pt} = &\hspace{-3pt}  \max_{a_t \in \mathcal{A}_t} \hspace{-1pt} \big(a_t \hspace{-1pt} + \hspace{-1pt} \mathbb{I}(x_t,a_t|m_t)\big) \hspace{-1pt} - \hspace{-1pt} \max_{a_t \in \mathcal{A}_t} \hspace{-1pt} \big(a_t + \mathbb{I}(a_t|m_t)\big) \nonumber \\
    = &\bar{c}_t(u_{0:t-1}) + \mathbb{I}(x_t|y_{0:t},u_{0:t-1}\big) - \bar{c}_t(u_{0:t-1}) \nonumber \\ 
    = &\mathbb{I}(x_t|y_{0:t}, u_{0:t-1}\big) = \mathbb{I}(x_t|m_t),
\end{align}
where, in the second equality, we use the fact that $\max_{a_t \in \mathcal{A}_t} \big(a_t + \mathbb{I}(x_t,a_t|m_t)\big) = \max_{a_t \in [[A_t|m_t]]} \big(a_t + \mathbb{I}(x_t|a_t, m_t)\big)$ and $[[A_t|m_t]] = \{\bar{c}_t(u_{0:t-1})\}$.
\end{proof}

As a direct consequence of Lemma \ref{lem_ad_case_4}, for a given realization of the memory $m_t \in \mathcal{X}_t$ at time $t$, the realization of the information state for a perfectly observed system is the function form of the indicator function $\mathbb{I}(X_t|m_t)$, where for all $x_t \in \mathcal{X}$,
\begin{gather} \label{ap_basic_1}
    \mathbb{I}(x_t|m_t) =
    \begin{cases}
         0, \quad \quad \;  \text{ if } x_t \in [[X_t|m_t], \\
        -\infty, \quad \text{ if } x_t \not\in [[X_t|m_t].
    \end{cases}
\end{gather}
From \eqref{ap_basic_1}, note that the functional form $\mathbb{I}(X_t|m_t) = \mathbb{I}(X_t|[[X_t|m_t]])$, and thus, at each time $t = 0,\dots,T$, given the realized memory $m_t \in \mathcal{M}_t$, it is sufficient to simply track the conditional range $[[X_t|m_t]]$ to derive the information state $r_t(X_t|m_t)$ for all $x_t \in \mathcal{X}$. This implies that $[[X_t|m_t]]$ satisfies all the properties of an information state.

\section*{Appendix D - Proof of Lemma 2}

We prove this result by considering two cases which are mutually exclusive but cover all the possibilities.
\textit{Case 1:} $\max_{\mathsf{x} \in \mathscr{X}}\big(f(\mathsf{x}) + r(\mathsf{x}) \big) \geq \max_{\mathsf{x} \in \mathscr{X}}\big(f(\mathsf{x}) + q(\mathsf{x})\big)$, which implies that $\big|\max_{\mathsf{x} \in \mathscr{X}}\big(f(\mathsf{x}) + r(\mathsf{x}) \big) - \max_{\mathsf{x} \in \mathscr{X}}\big(f(\mathsf{x}) + q(\mathsf{x})\big)\big| = \max_{\mathsf{x} \in \mathscr{X}}\big(f(\mathsf{x}) + r(\mathsf{x}) \big) - \max_{\mathsf{x} \in \mathscr{X}}\big(f(\mathsf{x}) + q(\mathsf{x})\big)$. We define a variable $\mathsf{x}^* \in \mathscr{X}^r$ such that $\mathsf{x}^* := \arg \max_{\mathsf{x} \in \mathscr{X}}\big(f(\mathsf{x}) + r(\mathsf{x}) \big)$ and a function $\psi^i: \mathscr{X} \to \mathscr{X}^i$ such that $\psi^i(\mathsf{x}) := \arg \min_{\tilde{\mathsf{x}} \in \mathscr{X}^i}d(\mathsf{x}, \tilde{\mathsf{x}})$ for each $i = r, q$. Then,
    $\max_{\mathsf{x} \in \mathscr{X}}\big(f(\mathsf{x}) + r(\mathsf{x}) \big) - \max_{\mathsf{x} \in \mathscr{X}}\big(f(\mathsf{x}) + q(\mathsf{x})\big)$ 
    $= f(\mathsf{x}^*) + r(\mathsf{x}^*) - \max_{\mathsf{x} \in \mathscr{X}}\big(f(\mathsf{x}) + q(\mathsf{x})\big)$
    $\leq f(\mathsf{x}^*) + r(\mathsf{x}^*) - f\big(\psi^q(\mathsf{x}^*)\big) - q\big(\psi^q(\mathsf{x}^*)\big)$
    $\leq L_f \cdot d(\mathsf{x}^*,\psi^q(\mathsf{x}^*)) + \big|r(\mathsf{x}^*) - q\big(\psi^q(\mathsf{x}^*)\big)\big|$
    $\leq L_f \cdot \mathcal{H}(\mathscr{X}^r, \mathscr{X}^q) + \max_{\mathsf{x} \in \mathscr{X}^r \cup \mathscr{X}^q}\big|r\big(\psi^r(\mathsf{x})\big) - q\big(\psi^q(\mathsf{x})\big)\big|$
    $\leq L_f \cdot \mathcal{R}\big(r,q\big) + \mathcal{R}\big(r,q\big),$
where, in the first inequality, we use the fact that $q(\psi^q(\mathsf{x}^*)) \neq - \infty$; in the second inequality, we use the Lipschitz continuity of $f$; in the third inequality, we use the definition of the Hausdorff metric from \eqref{H_met_def} and the fact that $\max_{x \in \mathcal{X}^r}|r(\mathsf{x}) - q(\psi^q(\mathsf{x}))| = \max_{\mathsf{x} \in \mathscr{X}^r \cup \mathscr{X}^q}|r(\psi^r(\mathsf{x})) - q(\psi^q(\mathsf{x}))|$; and in the fourth inequality, we use \eqref{eq_cost_metric} from Definition \ref{def_cost_metric}.
\textit{Case 2:} $\max_{\mathsf{x} \in \mathscr{X}}\big(f(\mathsf{x}) + r(\mathsf{x}) \big) < \max_{\mathsf{x} \in \mathscr{X}}\big(f(\mathsf{x}) + q(\mathsf{x})\big)$, where the result holds using the same arguments as Case 1.

\section*{Appendix E - Proof of Theorem 3}

For all $t=0,\dots,T$, let $m_t \in \mathcal{M}_t$ and $u_t \in \mathcal{U}$ be realizations of $M_t$ and $U_t$, respectively.
We prove both results by mathematical induction, starting with time step $T$. At $T$, we directly use \eqref{eq_ad_prelim} from Lemma \ref{lem_ad_prelim} and \eqref{ap_ad_2} from Definition \ref{def_ad_approx} to conclude that $|Q_T(m_T,u_T) - \hat{Q}_T(\hat{\sigma}_t(m_T),u_T)| \leq (L_{c_T} + 1) \cdot \epsilon_T$. 
Furthermore, minimizing both terms in the LHS of \eqref{thm_7_1} yields $|V_T(m_T) - \hat{V}_T(\hat{\sigma}_t(m_T))| \leq \max_{u_T \in \mathcal{U}}|Q_T(m_T,u_T) - \hat{Q}_T(\hat{\sigma}_t(m_T),u_T)| \leq (L_{c_T} + 1) \cdot \epsilon_T$. This forms the basis of our mathematical induction. Then, at each $t=0,\dots,T-1$, we consider the induction hypothesis $|V_{t+1}(m_{t+1}) - \hat{V}_{t+1}(\hat{\sigma}_{t+1}(m_{t+1}))| \leq \alpha_{t+1}$ and first prove \eqref{thm_7_1}. Using the triangle inequality,
    $\leq \big| \max_{x_t \in \mathcal{X}, m_{t+1} \in \mathcal{M}_{t+1}} ( {V}_{t+1}(m_{t+1}) +c_t(x_t,u_t) + r_t(x_t, m_{t+1}|m_t, u_t))$
    $- \max_{x_t \in \mathcal{X}, m_{t+1} \in \mathcal{M}_{t+1}}$ $( \hat{V}_{t+1}(\hat{\sigma}_{t+1}(m_{t+1})) +c_t(x_t,u_t) + r_t(x_t, m_{t+1}|m_t,$
    $u_t)) \big|  +  \big| \max_{x_t \in \mathcal{X}, \hat{\pi}_{t+1} \in \hat{\mathcal{P}}_{t+1}}( \hat{V}_{t+1}(\hat{\pi}_{t+1}) +c_t(x_t,u_t) + r_t(x_t, \hat{\pi}_{t+1}|m_t,    u_t) )$
    $- \; \; \max_{x_t \in \mathcal{X}, \hat{\pi}_{t+1} \in \hat{\mathcal{P}}_{t+1}}  ( \hat{V}_{t+1}(\hat{\pi}_{t+1}) +c_t(x_t,u_t) + r_t(x_t, \hat{\pi}_{t+1}|\hat{\sigma}_t(m_t), u_t)  ) \big|.$
Here, for the first term in the RHS,
    $\big| \max_{x_t \in \mathcal{X}, m_{t+1} \in \mathcal{M}_{t+1}} ( {V}_{t+1}(m_{t+1}) +c_t(x_t,u_t) + r_t(x_t, m_{t+1}|m_t, u_t) )$
    $- \max_{x_t \in \mathcal{X}, m_{t+1} \in \mathcal{M}_{t+1}}$ $\big ( \hat{V}_{t+1}(\hat{\sigma}_{t+1}(m_{t+1})) +c_t(x_t,u_t) + r_t(x_t, m_{t+1}|m_t, u_t)) \big|$
    $\leq \max_{x_t \in \mathcal{X}, m_{t+1} \in \mathcal{M}_{t+1}}\big|{V}_{t+1}(m_{t+1}) - \hat{V}_{t+1}(\hat{\sigma}_{t+1}(m_{t+1}))\big| \leq \alpha_{t+1},$
where in the second inequality, we use the induction hypothesis.
Furthermore, in the second term in the RHS, we directly use \eqref{eq_ad_prelim} from \ref{lem_ad_prelim} and \eqref{ap_ad_1} from Definition \ref{def_ad_approx} to conclude that
$\big| \max_{x_t \in \mathcal{X}, \hat{\pi}_{t+1} \in \hat{\mathcal{P}}_{t+1}} ( \hat{V}_{t+1}(\hat{\pi}_{t+1}) +c_t(x_t,u_t) + r_t(x_t, \hat{\pi}_{t+1}|m_t, u_t) )$
    $- \max_{x_t \in \mathcal{X}, \hat{\pi}_{t+1} \in \hat{\mathcal{P}}_{t+1}}$ $( \hat{V}_{t+1}(\hat{\pi}_{t+1}) +c_t(x_t,u_t) + r_t(x_t, \hat{\pi}_{t+1}|\hat{\sigma}_t(m_t), u_t) ) \big|$
    $\leq \big(2 L_t + 1 \big) \cdot \epsilon_t,$
where, $L_t = \max\{L_{\hat{V}_{t+1}}, L_{c_t}\}$ and $2L_t$ is the Lipschitz constant for the function $\phi(\hat{\pi}_{t+1}, x_t) := \hat{V}_{t+1}(\hat{\pi}_{t+1}) + c_t(x_t, u_t)$ with respect to the variables $(\hat{\pi}_{t+1}, x_t)$ for all $u_t \in \mathcal{U}$. Combining results for each term in the RHS completes the proof for \eqref{thm_7_1} at time $t$. Next, we prove \eqref{thm_7_2}. Using the definition of the value functions in the LHS of \eqref{thm_7_2}, $|V_t(m_t) - \hat{V}_t(\hat{\sigma}_t(m_t))| = |\min_{u_t \in \mathcal{U}}Q_t(m_t, u_t) - \min_{u_t \in \mathcal{U}}$
    $\hat{Q}_t(\hat{\sigma}_t(m_t), u_t) |
    \leq \max_{u_t \in \mathcal{U}}|Q_t(m_t, u_t) - \hat{Q}_t(\hat{\sigma}_t(m_t), u_t)|$
    $\leq \alpha_t,$
where in the second inequality, we use \eqref{thm_7_1}. This proves the induction hypothesis at time $t$. Thus, the results hold for all $t=0,\dots,T$ using mathematical induction.

\section*{Appendix F - Proof of Theorem 4}

We begin by recursively defining the value functions which compute the performance of the strategy $\boldsymbol{\hat{g}}$. 
For all $t=0,\dots,T-1$ and for each $\hat{\pi}_t \in \hat{\mathcal{P}}_t$ and $u_t \in \mathcal{U}$, let
$\hat{\Theta}_t (\hat{\pi}_t, u_t) := \max_{x_t \in \mathcal{X}, \hat{\pi}_{t+1} \in \hat{\mathcal{P}}_{t+1}} \big (\hat{\Lambda}_{t+1}(\hat{\pi}_{t+1}) + c_t(x_t,u_t) + r_t(x_t, \hat{\pi}_{t+1}|\hat{\pi}_t,u_t) \big)$ and
    $\hat{\Lambda}_t(\hat{\pi}_t):= \hat{\Theta}_t(\hat{\pi}_t, \hat{g}_t(\hat{\pi}_t))$;
where, at time $T$, $\hat{\Theta}_T(\hat{\pi}_T, u_T) := \max_{x_T \in \mathcal{X}}(c_T(x_T,u_T) + r_T(x_T|m_T,u_T))$ and $\hat{\Lambda}_{T}(\hat{\pi}_{T}) :=\hat{\Theta}_T(\hat{\pi}_t, \hat{g}_T(\hat{\pi}_T))$. Note that $\hat{\Theta}_t(\hat{\pi}_t, u_t) = \hat{Q}_{t}(\hat{\pi}_t, u_t)$ and  $\hat{\Lambda}_t(\hat{\pi}_t) = \hat{V}_t(\hat{\pi}_t),$
for all $t=0,\dots,T$, since $\hat{g}_t(\hat{\pi}_t)  =\arg\min_{u_t \in \mathcal{U}}\hat{Q}_t(\hat{\pi}_t,u_t)$.
Next, we use the triangle inequality in the LHS of \eqref{thm_8_3} at any $t$ to write 
    $\hat{Q}_t(\hat{\sigma}_t(m_t),u_t)| + |\hat{\Theta}_t(\hat{\sigma}_t(m_t), u_t) - \Theta_t(m_t, u_t)|$
    $\leq \alpha_t  + |\hat{\Theta}_t(\hat{\sigma}_t(m_t), u_t) - \Theta_t(m_t, u_t)|,$
where, in the second inequality, we use \eqref{thm_7_1} from Theorem \ref{thm_approx_ad_dp}. Then, to prove \eqref{thm_8_3}, it suffices to show that $|\hat{\Theta}_t(\hat{\sigma}_t(m_t), u_t) - \Theta_t(m_t, u_t)| \leq \alpha_t$. We can show this in addition to $|\hat{\Lambda}_t(\hat{\sigma}_t(m_t)) - \Lambda_t(m_t)| \leq \alpha_t$ for all $t=0,\dots,T$ using mathematical induction and following the same arguments as in Theorem \ref{thm_approx_ad_dp}.

\end{document}